\documentclass[review,12pt]{elsarticle}
\usepackage{amssymb}
\usepackage{amsmath}
\usepackage{amsthm}
\newtheorem{theorem}{Theorem}

\newdefinition{remark}{Remark}
\journal{Journal of Computational and Applied Mathematics}

\begin{document}

\begin{frontmatter}

\title{Numerical Integration as a Finite Matrix Approximation to
       Multiplication Operator} 

\author[add]{Juha Sarmavuori\corref{cor}}
\ead{juha.sarmavuori@aalto.fi}
\address[add]
        {
           Department of Electrical Engineering and Automation,
          Aalto University,
          P.O. Box 12200, 
          FI-00076 Aalto,
          Finland
        }
\cortext[cor]
        {
          Corresponding author
        }
\author[add]{Simo S\"arkk\"a}
\ead{simo.sarkka@aalto.fi}

\begin{abstract}
  In this article, numerical integration is formulated as evaluation of a matrix function of a matrix
  that is obtained as a projection of the multiplication operator on a finite-dimensional basis.
  The idea is to approximate the continuous spectral representation of a 
  multiplication operator on a Hilbert space with a discrete spectral 
  representation of a Hermitian matrix. The Gaussian quadrature is shown to be a special
  case of the new method. The placement of the nodes of numerical integration
  and convergence of the new method are studied.
\end{abstract}


\end{frontmatter}


%
\section{Introduction}
This article is concerned with numerical integration which is an important task that
arises in almost all fields of science and engineering. 
We develop a method for numerical integration for a situation, where
it is possible to decompose the integrand into an outer and inner function
$f(g(\boldsymbol{x}))$ and to find solutions to certain integrals involving the function
$g$. The solvable integrals are elements of an infinite 
matrix $\boldsymbol{M}$ that corresponds to the operation of multiplying with the function $g$. The numerical
integration then reduces to approximate computation of the matrix function $f(\boldsymbol{M})$.

The approach may seem complicated at the first sight, but we show
that it is feasible at least in some cases. In fact, in hindsight, it can be interpreted to have
been used for a couple of centuries in Gaussian quadrature rules where the inner function is simply $g(x)=x$ and
the matrix corresponding to the operation of multiplying with $g$ is the infinite
tridiagonal Jacobi matrix
\cite{Stone:1932,Akhiezer:1965,Gautschi:1996,Gautschi:2002,Gautschi:2004,Golub+Meurant:2009,Simon:2015}.
In that well known special case, the matrix approximation leads
to the Golub-Welsch algorithm \cite{Golub+Welsch:1969,Gautschi:2002}
and its variations \cite{Laurie:2001,Golub+Meurant:2009}.

Recently, this idea has been applied on integrals on the unit circle of the complex plane when the
basis functions are rational functions
\cite{Velazquez:2008, Cantero+Cruz-Barroso+Gonzalez-Vera:2008, Cruz-Barroso+Delvaux:2009,Bultheel+Cantero:2009,
Bultheel+Gonzalez-Vera+Hendriksen+Njastad:2010,Bultheel+Cantero+Cruz-Barroso:2015}.
In this article, we consider more general integration
rules for the $d$-dimensional space and general orthonormal functions but restrict the integrals
to the real line.

Given the matrix multiplication operator interpretation of numerical integration,
we can use theoretical results for multiplication operators from other contexts and apply them to numerical integration.
Matrix approximation of multiplication operators of arbitrary complex functions 
in general orthonormal function bases 
was considered in \cite{Morrison:1995}. Although the aim in 
\cite{Morrison:1995} was not primarily on numerical 
integration, some of the convergence results apply to special cases where the 
integration weights are equal.
We adopt notation from there and a fundamental theorem \cite[Theorem 3.4]{Morrison:1995} about 
the placement of the nodes in numerical integration.
The matrix approximation of multiplication operator has also been used as a starting
point for finding nodes for generalized Gaussian quadratures
\cite{Vioreanu+Rokhlin:2014,Vioreanu:2012}.

The contribution of this article is to formulate a general class of numerical integration problems as matrix functions
of finite-dimensional approximations of multiplication operators. We also study the convergence of
the resulting method as well as show that the nodes of the method are located in a closed interval determined
by the infimum and supremum of the inner function $g(\boldsymbol{x})$. 
The usual approach on Gaussian quadrature is polynomial interpolation
\cite{Golub+Welsch:1969,Davis+Rabinowitz:1984,Gautschi:1996,Gautschi:2002,Gautschi:2004,Laurie:2001,Golub+Meurant:2009},
but
the identification of the Jacobi matrix as a multiplication operator allows us to make
three generalizations:
\begin{enumerate}
\item We can use other inner functions than just $g(x)=x$. The inner function can also be
a scalar function of a multidimensional variable.
\item 
A matrix function is also an approximation to a multiplication operator and it can be 
used
in approximations of integrals that involve products of different functions.
\item We can generalize the notion of classical Gaussian quadrature to any 
basis functions, not just polynomials.
The usual approach to generalized Gaussian quadrature is based on a set of 
non-linear
equations \cite{Ma+Rokhlin+Wandzura:1996,Vioreanu+Rokhlin:2014}.
\end{enumerate}

The main specialized operator-theoretic tools that we use are multiplication operators
and their infinite matrix representations. Both subjects are well presented in
\cite{Weidmann:1980}.
The books \cite{Segal+Kunze:1978,Reed+Simon:1981,Simon:2015} lack only in the infinite matrices
which can be found in \cite[Chapter 3, section 1]{Stone:1932} or
\cite[Sections 26 and 47]{Akhiezer+Glazman:1993}. 
The book \cite{Segal+Kunze:1978} also presents different traditional definitions and ideas of integration
from operator theoretic perspective and builds a completely operator theory based
algebraic integration theory (see also \cite{Segal:1965}).

The paper is organized as follows. The key concepts are introduced
in Section \ref{sec:prelim}.
Our main theoretical results are in
Section \ref{sec:main}.
Experimental
results are presented in Section \ref{sec:experiments}, and the conclusions follow in
Section \ref{sec:conclusions}.

\section{Preliminaries}
\label{sec:prelim}

The purpose of this paper is to numerically solve an integral of the form
\begin{equation}
  \int_{\Omega}f(g(\boldsymbol{x}))\,w(\boldsymbol{x})\,d\boldsymbol{x},
  \label{eq:first_int}
\end{equation}
where $\Omega\subset\mathbb{R}^d$,
$w :~ \Omega \mapsto[0,\infty)$,
$g :~ \Omega\mapsto\mathbb{R}$,
$f :~ \overline{g(\Omega)}\mapsto\mathbb{R}$,
and $\overline{g(\Omega)}$ is the closure of the image of the inner 
function $g$.
For the ease of exposition,
we normalize $w(\boldsymbol{x})$ so that
$\int_{\Omega}w(\boldsymbol{x})\,d\boldsymbol{x}=1$.

Because the image of the inner function $g$ is one dimensional,
the integral over the outer function $f$ is one dimensional integral with 
respect
to measure $\mu$ defined as
  $\mu(f) = \int_{g(\Omega)} f(y)\,d\mu(y) =
  \int_\Omega f(g(\boldsymbol{x}))\,w(\boldsymbol{x})\,d\boldsymbol{x}$.
For example,
$g(\boldsymbol{x})=x_1+x_2$ or $g(\boldsymbol{x})=x_1\,x_2$ can be
simple enough functions that it is possible to solve the involved integrals
in closed form while the composite function $f(g(\boldsymbol{x}))$ is
more complicated and requires numerical approach. It is worth to
note that the numerical subproblem is only one dimensional. 

We also define a Hilbert space
$\mathcal{L}_w^2(\Omega)$ with the inner
product
  $\langle \phi, \psi \rangle = \int_\Omega \overline{\phi(\boldsymbol{x})}\,\psi(\boldsymbol{x})\,
w(\boldsymbol{x})\,d\boldsymbol{x}$.
Let $g$ be a bounded function.
We define a multiplication operator 
  $\mathsf{M}[g]:~\mathcal{L}_w^2(\Omega)\mapsto \mathcal{L}_w^2(\Omega)$
almost everywhere pointwise as
$(\mathsf{M}[g]\,\phi)(\boldsymbol{x})=g(\boldsymbol{x})\,\phi(\boldsymbol{x})$.
Thus, the effect of the operator $\mathsf{M}[g]$ on a function $\phi$ is
multiplication with the function $g$ and the equality \label{eq:pointwise-mg}
may not hold for a set of points that does not contribute anything on an
integral over the weight function $w(\boldsymbol{x})$.

We define a bounded function of a multiplication operator as
a multiplication operator of the composite function
$(f(\mathsf{M}[g])\,\phi)(\boldsymbol{x})=
f(g(\boldsymbol{x}))\,\phi(\boldsymbol{x})=(\mathsf{M}[f(g)]\,\phi)(\boldsymbol{x})$.
With this notion, we can rewrite the integral as
\begin{equation}
\label{eq:innerproduct-integral}
\int_{\Omega}f(g(\boldsymbol{x}))\,w(\boldsymbol{x})\,d\boldsymbol{x}
=\langle 1, f(\mathsf{M}[g])\,1\rangle.
\end{equation}

In our practical examples, the Hilbert space $\mathcal{L}_w^2(\Omega)$
is separable, that is, there
is a countable set of orthonormal basis functions 
$\phi_0(\boldsymbol{x}),\phi_1(\boldsymbol{x}),\ldots$ that are dense
in $\mathcal{L}_w^2(\Omega)$. 
We define an orthogonal projection operator $\mathsf{P}_n$ as
\begin{equation}
\label{eq:projection}
\mathsf{P}_n\,f = \sum_{k=0}^n \langle \phi_k, f \rangle \, \phi_k.
\end{equation}
We can project a multiplication operator
to a subspace of the Hilbert space by first projecting the
operand and then projecting the result of the operation again. The projected
multiplication operator is thus $\mathsf{P}_n\,\mathsf{M}[g]\,\mathsf{P}_n$.
When the projection is to a finite dimensional subspace, the projected multiplication
operator can be represented with a finite matrix $\boldsymbol{M}_n[g]$ with elements
\begin{equation*}
\left[
  \boldsymbol{M}_n[g]
\right]_{i,j}
=\langle 
\phi_i,\mathsf{M}[g]\,\phi_j \rangle,~i,j=0,1,2,\ldots,n.
\end{equation*}
We start the matrix element indexing at 0.

As the matrix $\boldsymbol{M}_n[g]$ is Hermitian, it has the eigenvalue decomposition
\begin{equation*}
\boldsymbol{M}_n[g]=\boldsymbol{U}\,
\left[
  \begin{array}{cccc}
    \lambda_0 &           &        & \\
              & \lambda_1 &        & \\
              &           & \ddots & \\
              &           &        & \lambda_n
  \end{array}
\right]\,
\boldsymbol{U}^*,
\end{equation*}
where $\boldsymbol{U}$ is a unitary matrix and $\lambda_i \in \mathbb{R}$ are the 
eigenvalues of 
$\boldsymbol{M}_n[g]$. A function of a Hermitian matrix can be defined
with the use of the eigenvalue decomposition as (see \cite[Chapter 1.2]{Higham:2008}
or \cite[Corollary 11.1.2]{Golub+vanLoan:1996})
\begin{equation*}
f(\boldsymbol{M}_n[g])=\boldsymbol{U}\,
\left[
  \begin{array}{cccc}
    f(\lambda_0) &              &        & \\
                 & f(\lambda_1) &        & \\
                 &              & \ddots & \\
                 &              &        & f(\lambda_n)
  \end{array}
\right]\,
\boldsymbol{U}^*.
\end{equation*}
Usually, for a finite matrix $f(\boldsymbol{M}_n[g])\neq\boldsymbol{M}_n[f(g)]$.

If the orthonormal basis functions span the whole Hilbert space
$\mathcal{L}_w^2(\Omega)$, the space is
isomorphic with $\ell^2$, the space of the square summable sequences or infinite column vectors.
The basis functions $\phi_i$ are isomorphic with infinite column vectors $\boldsymbol{e}_i$, that is, the basis vectors of $\ell^2$ that have a 1 in $i$th component
and 0 in other components. 
We denote an isomorphism by $\simeq$.
Generally, we have the following isomorphisms:
\begin{eqnarray}
\label{eq:space-isomorphism}
  \mathcal{L}_w^2(\Omega) & \simeq & \ell^2\\
\label{eq:basis-isomorphism}
\phi_i & \simeq & \boldsymbol{e}_i\\
\label{eq:vector-isomorphism}
\psi & \simeq &
\left[
\begin{array}{c}
  \langle \phi_0, \psi\rangle \\
  \langle \phi_1, \psi\rangle \\
  \vdots
\end{array}
\right]\\
\label{eq:matrix-isomorphism}
\mathsf{M}[\operatorname{g}]\,\psi &\simeq & \boldsymbol{M}_\infty[g]\,
\left[
\begin{array}{c}
  \langle \phi_0, \psi\rangle \\
  \langle \phi_1, \psi\rangle \\
  \vdots
\end{array}
\right]\\
\label{eq:full-isomorphism}
f(\mathsf{M}[g]) & \simeq & f(\boldsymbol{M}_\infty[g])\\
\label{eq:1-element-isomorphism}
\langle 1, f(\mathsf{M}[g])\, 1\rangle & = & 
\boldsymbol{e}_0^\top\,f(\boldsymbol{M}_\infty[g])\,\boldsymbol{e}_0.
\end{eqnarray}
Isomorphisms \eqref{eq:vector-isomorphism} and \eqref{eq:basis-isomorphism} follow
from \eqref{eq:space-isomorphism} which is equivalent to separability.
Sufficient conditions for \eqref{eq:matrix-isomorphism} are
\eqref{eq:space-isomorphism} and that $\mathsf{M}[g]$ has
an infinite matrix representation (see
\cite[Theorems 3.4 and 3.5]{Stone:1932} or 
\cite[Section 26 and 47]{Akhiezer+Glazman:1993})
for which boundedness of $g$ is sufficient. For
\eqref{eq:full-isomorphism} sufficient conditions are
\eqref{eq:matrix-isomorphism} and 
that $\mathsf{M}[f(g)]$ has an infinite matrix representation.
The last isomorphism \eqref{eq:1-element-isomorphism}
is the most important one for numerical integration due to
the identity \eqref{eq:innerproduct-integral}. It is
not only an isomorphism, but also an equality because
in both Hilbert spaces the quantity is a real scalar. The isomorphism
\eqref{eq:1-element-isomorphism}
can still
hold even
if the isomorphism \eqref{eq:full-isomorphism} does not.

In the case $d=1$,
$g(x)=x$, and polynomial $\phi_i$, the infinite matrix
$\boldsymbol{M}_\infty[g]$ is a tridiagonal Jacobi matrix $\boldsymbol{J}_\infty$
\cite{Simon:2008,Simon:2015}. 
In that special case, without the isomorphism considerations,
an approximation of \eqref{eq:1-element-isomorphism} has been
recognized as a Gaussian quadrature rule in a form 
$\int_\Omega f(x)\,w(x)\,dx \approx
\boldsymbol{e}_0^\top\,f(\boldsymbol{J}_n)\,\boldsymbol{e}_0$
where $\boldsymbol{J}_n$ is a finite truncation of $\boldsymbol{J}_\infty$
\cite[Equation (2.10)]{Gautschi:2002},
\cite[Equation (3.1.8)]{Gautschi:2004},
\cite[Theorem 6.6]{Golub+Meurant:2009}.
From the isomorphism considerations, it is easy to generalize the inner function
to something else than $g(x)=x$ and likewise the basis functions to any orthonormal functions
instead of polynomials. Since the matrix approximation approximates a multiplication
operator, it is quite natural to use it to approximate multiplication
with a function. In that case, other matrix elements, not just $(0,0)$ element, are used as well as we will show in the following.

\section{Main results}
\label{sec:main}
The  above  discussion  suggests  a  method  for  approximating  an
integral  of  the  form \eqref{eq:first_int} as follows. 
Take orthonormal basis functions $\phi_0=1,\phi_1,\phi_2,\ldots,\phi_n$ and
for all $i,j=0,1,2,\ldots n$ compute the matrix elements
\begin{equation*}
[\boldsymbol{M}_n[g]]_{i,j}=
\int_\Omega g(\boldsymbol{x})\,\overline{\phi_i(\boldsymbol{x})}\,
\phi_j(\boldsymbol{x})\,
w(\boldsymbol{x})\,d\boldsymbol{x}.
\end{equation*}
We can then approximate the integral \eqref{eq:first_int} numerically by
\begin{equation*}
\int_\Omega f(g(\boldsymbol{x}))\,w(\boldsymbol{x})\,d\boldsymbol{x}
\approx [f(\boldsymbol{M}_n[g])]_{0,0}.
\end{equation*}
This formula is a quadrature rule in the traditional sense since
\begin{equation*}
[f(\boldsymbol{M}_n[g])]_{0,0}=
\boldsymbol{e}_0^\top\,f(\boldsymbol{M}_n)\,\boldsymbol{e}_0
=
\sum_{i=0}^n f(\lambda_i)\,
\boldsymbol{e}_0^\top\,\boldsymbol{u}_i\,\boldsymbol{u}_i^*\,\boldsymbol{e}_0
=\sum_{i=0}^n |[\boldsymbol{u}_i]_0|^2\,f(\lambda_i),
\end{equation*}
where $\lambda_i$ and $\boldsymbol{u}_i$ are the eigenvalues and 
unit length eigenvectors of
$\boldsymbol{M}_n[g]$ and $[\boldsymbol{u}_i]_0$ is the 0th component of the
eigenvector $i$.
In the quadrature terminology, $\lambda_i$ are the nodes or abscissas and 
$|[\boldsymbol{u}_i]_0|^2$ are the weights. We see that the
weights are all positive which is important for convergence and stability
of the quadrature and it is also true for Gaussian
quadrature rules \cite[Theorem 1.46]{Gautschi:2004}.

We can also use other matrix elements than $i=j=0$ to approximate integrals
$\int_\Omega f(g(\boldsymbol{x}))\,\overline{\phi_i(\boldsymbol{x})}\,
\phi_j(\boldsymbol{x})\,w(\boldsymbol{x})\,d\boldsymbol{x}
\approx [f(\boldsymbol{M}_n[g])]_{i,j}$.
If vector $\boldsymbol{v}$ contains the Fourier series coefficients of $\psi(\boldsymbol{x})$, that is, 
$v_i=\langle \phi_i,\psi\rangle$ then we can use it to approximate
$\int_\Omega f(g(\boldsymbol{x}))\,
\psi(\boldsymbol{x})\,w(\boldsymbol{x})\,d\boldsymbol{x}
\approx [f(\boldsymbol{M}_n[g])\,\boldsymbol{v}]_0$.
Since the matrix $\boldsymbol{M}_n[g]$ approximates the multiplication operator,
it can be used to approximate the integrals of the inner product form
\begin{align*}
\int_\Omega f_1(g_1(\boldsymbol{x}))\,f_2(g_2(\boldsymbol{x}))\,
w(\boldsymbol{x})\,d\boldsymbol{x} &\approx 
[f_1(\boldsymbol{M}_n[g_1])\,f_2(\boldsymbol{M}_n[g_2])]_{0,0}
\end{align*}
or more generally
\begin{equation*}
\int_\Omega 
\prod_{i=0}^m
f_i(g_i(\boldsymbol{x}))
w(\boldsymbol{x})\,d\boldsymbol{x} \approx
\left[ 
\prod_{i=0}^m
f_i(\boldsymbol{M}_n[g_i])
\right]_{0,0}.
\end{equation*}
Here, we must notice that the value of the numerical approximation depends on the
order of the matrices in the product. This is because the matrix
approximations do not commute with respect to multiplication although the multiplication operators do.

We define the sum and product of the multiplication operators pointwise
\begin{align*}
  ((\mathsf{M}[f]+\mathsf{M}[g])\,\phi)(\boldsymbol{x})
  = 
  (f(\boldsymbol{x})+g(\boldsymbol{x}))\,\phi(\boldsymbol{x})
  &=
  (\mathsf{M}[f+g]\,\phi)(\boldsymbol{x}),\\
  ((\mathsf{M}[f]\,\mathsf{M}[g])\,\phi)(\boldsymbol{x})
  = 
  f(\boldsymbol{x})\,g(\boldsymbol{x})\,\phi(\boldsymbol{x})
  &=
  (\mathsf{M}[f\,g]\,\phi)(\boldsymbol{x}).
\end{align*}
By this definition, the multiplication operators clearly commute. We see that
for a finite matrix approximation,
commutativity for the product is not preserved in the homomorphism
while for the sum it is.

Two effects of the non-commutativity are that the value of the approximation
depends on the order of the terms in the product and the product matrix
is not necessarily Hermitian. 
A non-Hermitian matrix can also be non-diagonalizable and the matrix
function may have to involve derivatives.
It is also possible to symmetrize the product of matrices by computing
the product in two opposite orders and taking the average, that is, the matrix
\begin{equation*}
\frac{1}{2}\,(
f_1(\boldsymbol{M}_n[g_1])\,f_2(\boldsymbol{M}_n[g_2])\,
f_3(\boldsymbol{M}_n[g_3])+
f_3(\boldsymbol{M}_n[g_3])\,f_2(\boldsymbol{M}_n[g_2])\,
f_1(\boldsymbol{M}_n[g_1]))
\end{equation*}
is Hermitian and we can approximate
\begin{align*}
  &\int_{\Omega} f_4(f_1(g_1(\boldsymbol{x}))\,f_2(g_2(\boldsymbol{x}))\,
  f_3(g_3(\boldsymbol{x})))
w(\boldsymbol{x})\,d\boldsymbol{x}
\approx\\
&
\left[
f_4
\left(
\frac{\scriptstyle
  f_1(\boldsymbol{M}_n[g_1])\,f_2(\boldsymbol{M}_n[g_2])\,
  f_3(\boldsymbol{M}_n[g_3])+
  f_3(\boldsymbol{M}_n[g_3])\,f_2(\boldsymbol{M}_n[g_2])\,
  f_1(\boldsymbol{M}_n[g_1])
}{2}
\right)
\right]_{0,0}.
\end{align*}
Basically, we can replace functions in any formula with matrices and an approximation
for the integral is given by the upper left corner of the final matrix.

\begin{remark}
  We can also define the matrix $\boldsymbol{M}_n[g]$ in terms of non-or\-tho\-nor\-mal functions
  as in \cite[Section 4]{Golub+Welsch:1969} and \cite[Chapter 5.2]{Golub+Meurant:2009} for Gaussian quadrature.
  Given arbitrary linearly independent but non-orthonormal functions
  $\tilde{\phi}_0,\tilde{\phi}_1,\ldots,\tilde{\phi}_n$,
  the Gram matrix has elements
  $[\boldsymbol{G}]_{i,j}=\langle \tilde{\phi}_i,\tilde{\phi}_j \rangle.$
  We can define the matrix as
  \begin{equation}
    \label{eq:non-orthonormal-matrix}
    \boldsymbol{M}_n[g]=(\boldsymbol{R}^{-1})^*\,\widetilde{\boldsymbol{M}}_n[g]\,\boldsymbol{R}^{-1},
  \end{equation}
  where $\boldsymbol{R}$ is the Cholesky decomposition of the Gram matrix, that is,
  $\boldsymbol{R}^*\,\boldsymbol{R} = \boldsymbol{G}$ and
$    \left[\widetilde{\boldsymbol{M}}_n[g]\right]_{i,j}=\langle\tilde{\phi}_i,g\,\tilde{\phi}_j\rangle$.
  It was noted already in \cite{Golub+Welsch:1969,Golub+Meurant:2009} that the Gram matrix is ill-conditioned and
  \eqref{eq:non-orthonormal-matrix} is not suitable for numerical computations.
  However, when it is possible to compute $\boldsymbol{M}_n[g]$ in closed form,
  \eqref{eq:non-orthonormal-matrix} is faster on symbolic computations than orthonormalizing the basis
  functions with symbolic computations. For stable numerical
  computations, the algorithms in \cite{Yanagisawa+Ogita+Oishi:2014, Ozaki+Ogita+Oishi+Rump:2012} could be used
  for the inverse of Cholesky factorization and multiplication.
\end{remark}

\subsection{The range of the nodes}
\label{sec:range}

The nodes or abscissas of a numerical integration rule are the points $\boldsymbol{x}_i$ where the
integrand function is evaluated. In our approach the nodes are the eigenvalues
of the matrix $\boldsymbol{M}_n[g]$. 
In numerical integration we want to avoid nodes that are outside the domain of the function.
The following theorem gives conditions which ensure that the nodes are within 
the domain of
function $f$.
\begin{theorem}
\label{thm:realrange}
For a bounded real function $g$, the eigenvalues of
$\boldsymbol{M}_n[g]$ are 
 in the closed interval
 $[\inf g,\sup g]$.
\end{theorem}
\begin{proof}
  The complex version of the theorem \cite[Theorem 3.4]{Morrison:1995} (or
  \cite[Theorem 3.4.2]{Vioreanu:2012} for bounded and additionally 
  continuous $g$)
  states that the eigenvalues are
  in the convex hull of the essential range of the function $g$. For a real
  function, the convex hull of the essential range is the interval between
  the essential infimum and the essential supremum which in its turn is
  inside $[\inf g, \sup g]$.
\end{proof}
From the theorem we see that if the range $g(\Omega)$  is convex,
that is, the range does not have any holes, then everything is fine and
the nodes are in $\overline{g(\Omega)}$. 
If $g(\Omega)$ 
has holes, then it is possible to extend the definition
of $f$ as having the value of $0$ on the holes of $\overline{g(\Omega)}$ or by dividing $\Omega$
into parts $\Omega_i$ so that $\overline{g(\Omega_i)}$ is convex for each
$i$.

Theorem \ref{thm:realrange} is also well known property of Gaussian quadrature
\cite[Theorem 1.46]{Gautschi:2004}. Another well known property of the Gaussian
quadrature is the interlacing property of the nodes.
\begin{theorem}
Eigenvalues of $\boldsymbol{M}_n[g]$ and $\boldsymbol{M}_{n+1}[g]$ interlace,
that is, let $\{\alpha_i\}_{i=0}^n$ be eigenvalues of $\boldsymbol{M}_n[g]$ and
$\{\beta_i\}_{i=0}^{n+1}$ be eigenvalues of $\boldsymbol{M}_{n+1}[g]$
ordered from smallest to largest,
then $\beta_i\leq \alpha_i$ for $i=0,1,\ldots,n$ and $\alpha_n\leq \beta_{n+1}$.
\end{theorem}
\begin{proof}
This follows directly from well known Cauchy's interlacing theorem
\cite[Corollary III.1.5]{Bhatia:1997} or \cite[Theorem 1.3.5]{Simon:2015}.
\end{proof}
However, for Gaussian quadrature the inequality is strict, that is,
$\beta_i < \alpha_i$ and $\alpha_n < \beta_{n+1}$
\cite[Theorem 1.20]{Gautschi:2004}.
This demonstrates that when the basis functions
$\phi_i$ are not polynomials or the inner function $g(x)\neq x$, some
properties of Gaussian quadrature may hold in similar, but not necessarily
in exactly the same form.

\subsection{Convergence for bounded functions}
In this section, we analyze the convergence of the new method for bounded functions.

A basic requirement for the convergence is that the multiplication
operator has an infinite matrix representation. For a bounded
function $g$, the multiplication operator $\mathsf{M}[g]$ is also bounded and
it has an infinite matrix representation $\boldsymbol{M}_\infty[g]$ if the
Hilbert space is separable \cite[Theorem 3.5]{Stone:1932},
\cite[Section 26]{Akhiezer+Glazman:1993}.
Thus, for a bounded function $g$, we have the isomorphism
\eqref{eq:matrix-isomorphism}
for all $\psi\in\mathcal{L}_w^2(\Omega)$.

For bounded operators, we can use the concept of \textit{strong convergence}.
We say that bounded operators $\mathsf{A}_n$ converge strongly to
a bounded operator $\mathsf{A}$ if for any $\phi$ in a Hilbert space
$\|(\mathsf{A}_n-\mathsf{A})\,\phi\|\rightarrow 0$ as $n\rightarrow \infty$.
Then we express this as
$\mathsf{A}_n\xrightarrow{s}\mathsf{A}$.

In a separable Hilbert space with dense basis functions
$1,\phi_1,\phi_2,\ldots$, we can use the orthogonal projection operator
$\mathsf{P}_n$ of \eqref{eq:projection} and we see that 
for a bounded function $g$ we have
$  \mathsf{P}_n\,\mathsf{M}[g]\,\mathsf{P}_n
  \xrightarrow{s}\mathsf{M}[g]$.
This is equivalent to
\begin{equation*}
  \left[
    \begin{array}{ccc}
      \boldsymbol{M}_n[g] & 0      & \ldots\\
      0                   & 0      & \ldots\\
      \vdots              & \vdots & \ddots
    \end{array}
    \right]
  \xrightarrow{s}
  \boldsymbol{M}_\infty[g].
\end{equation*}

For a bounded function of a bounded multiplication operator,
we have the following theorem.
\begin{theorem} \label{thm:convergence1}
  Let $\mathcal{L}_w^2(\Omega)$ be a separable Hilbert space with dense
  set of basis functions $\phi_0=1,\phi_1,\ldots$.
  Let $g$ be a bounded real function. Let
  $\mathsf{E}(t)$ be the spectral family of $\mathsf{M}[g]$.
  Let $f$ be a bounded piecewise continuous function on 
  $\mathbb{R}$.
  Let the set of discontinuities of $f$ be $K$
  and closure of discontinuities $\overline{K}$.
  Let
    $\int_{\overline{K}} d\mathsf{E}(t)=\mathsf{0}$,
  that is, the
  discontinuities of $f(t)$ are not
  discontinuities of $\mathsf{E}(t)$ and $K$ is not dense in any
  subinterval of $\overline{g(\Omega)}$. Then
  \begin{equation*}
    f(\mathsf{P}_n\,\mathsf{M}[g]\,\mathsf{P}_n)
    \xrightarrow{s}
    f(\mathsf{M}[g])=\mathsf{M}[f(g)]
  \end{equation*}
  or equivalently
  \begin{equation*}
    \left[
      \begin{array}{ccc}
        f(\boldsymbol{M}_n[g]) & 0      & \ldots\\
        0                      & f(0)   & \ldots\\
        \vdots                 & \vdots & \ddots
      \end{array}
      \right]
    \xrightarrow{s}
    f(\boldsymbol{M}_\infty[g])=\boldsymbol{M}_\infty[f(g)].
  \end{equation*}
\end{theorem}
\begin{proof}
  See \cite[Theorem 2.6]{Bade:1954} for a much more general proof
  that holds for nets of self-adjoint operators in Banach spaces.
\end{proof}
\begin{remark}
  Spectral family of a multiplication operator $\mathsf{M}[g]$ is
  defined by a characteristic function
  \begin{equation*}
\chi_{\{\boldsymbol{x}:~g(\boldsymbol{x})\leq t\}}(\boldsymbol{x}) = 
\left\{
\begin{array}{ll}
0, & g(\boldsymbol{x}) > t,\\
1, & g(\boldsymbol{x}) \leq t
\end{array}
\right.
  \end{equation*}
  as
  $\mathsf{E}(t)=\mathsf{M}[\chi_{\{\boldsymbol{x}:~g(\boldsymbol{x})\leq t\}}]$
  \cite[Section 7.2, Example 1]{Weidmann:1980}.
  The only discontinuities of $\mathsf{E}(t)$ are 
  the eigenvalues of $\mathsf{M}[g]$, that is, the values $\lambda$ that satisfy
  $\mathsf{M}[g]\,\phi=\lambda\,\phi$ for some function $\phi$,
  \cite[Theorem 7.23]{Weidmann:1980}.
\end{remark}
\begin{remark}
This theorem covers, for example, piecewise continuous functions with
finite number of discontinuities. It does not cover all Riemann integrable
functions. For example, let $g(\Omega)=[0,1]$ and let $f$ be 
Thomae's function. Then $f$ is Riemann integrable, but
is not covered by this theorem since $K = [0,1] \cap \mathbb{Q}$ and
$\overline{K}=[0,1]$ \cite[Example 7.1.7]{Bartle+Sherbert:2011}.
\end{remark}
\begin{remark}
  The isomorphism  \eqref{eq:full-isomorphism} also follows from
  this theorem as $n\rightarrow \infty$
  for functions $g$ and $f$ that satisfy the conditions.
\end{remark}
The strong convergence also has nice addition and multiplication
properties.
\begin{theorem}
  Let $\mathsf{A}_n,\mathsf{B}_n,\mathsf{A},\mathsf{B}$ be bounded operators
  on a Hilbert space so that
    $\mathsf{A}_n \xrightarrow{s}\mathsf{A}$ and
    $\mathsf{B}_n \xrightarrow{s}\mathsf{B}$.
  Then
  \begin{align*}
    \mathsf{A}_n+\mathsf{B}_n &\xrightarrow{s}\mathsf{A}+\mathsf{B}, \\
    \mathsf{A}_n\,\mathsf{B}_n &\xrightarrow{s}\mathsf{A}\,\mathsf{B}.
  \end{align*}
\end{theorem}
\begin{proof}
  For the sum, see \cite[Chapter 4.9, problem 2]{Kreyszig:1989}.
  For the product, see 
  \cite[Chapter III, Lemma 3.8]{Kato:1995},
  \cite[Exercise 4.20]{Weidmann:1980}, or 
  \cite[Chapter 2.1, problem 5]{Simon:2015}.
\end{proof}
Thus, for instance, by the properties
of the strong convergence, we can prove that
for functions
$f_1,f_2,f_3$ and $g_1,g_2$
satisfying conditions of Theorem~\ref{thm:convergence1}, so that, 
$f_1$ is continuous on the
eigenvalues of $\mathsf{M}[g_1]$, $f_2$ on eigenvalues of $\mathsf{M}[g_2]$,
and $f_3$ on eigenvalues of $\mathsf{M}[f_1(g_1)\,f_2(g_2)]$, we have
\begin{align*}
  f_3\left(
  \frac{\scriptstyle f_1(\boldsymbol{M}_n[g_1])\,f_2(\boldsymbol{M}_n[g_2])
    +
    f_2(\boldsymbol{M}_n[g_2])
    \,
    f_1(\boldsymbol{M}_n[g_1])}{2}
  \right)%
  \xrightarrow{s}
  \boldsymbol{M}_\infty[f_3(f_1(g_1)\,f_2(g_2))].
\end{align*}

Strong convergence implies \textit{weak convergence}, that is,
for any vectors $\phi,\psi$ in a Hilbert space we have
$  \langle\psi,\mathsf{A}_n\,\phi\rangle
  \rightarrow\langle \psi,\mathsf{A}\,\phi\rangle$.
Weak convergence also covers the convergence of the $(0,0)$ matrix
element to the integral by selecting $\phi=\psi=\boldsymbol{e}_0$ which gives the following.
\begin{theorem} \label{thm:convergence2}
Let $f$ and $g$ be as in Theorem~\ref{thm:convergence1}, then
as $n\rightarrow \infty$
\begin{equation*}
  [f(\boldsymbol{M}_n[g])]_{0,0}=
  \boldsymbol{e}_0^\top\,f(\boldsymbol{M}_n[g])\,\boldsymbol{e}_0
  \rightarrow 
  \int_\Omega f(g(\boldsymbol{x}))\,w(\boldsymbol{x})\,d\boldsymbol{x}.
\end{equation*}
\end{theorem}

\section{Numerical results}
\label{sec:experiments}
As an example, we consider integration on interval $[0,1]$
with the weight function $w(x)=1$ and the system of
functions
\begin{equation*}
  1, x^{\frac{1}{3}}, x, x^{1+\frac{1}{3}}, x^2, x^{2+\frac{1}{3}}, \ldots,
  x^{n-1}, x^{n-1+\frac{1}{3}}
\end{equation*}
that were used in \cite{Ma+Rokhlin+Wandzura:1996}. We compare our method to generalized
Gaussian quadrature. For that method the functions determine the quadrature
rule so that it is exact for the $2\,n$ functions. For the
proposed matrix method the functions determine the quadrature so that
they span a subspace where the
multiplication operator is projected. Although the methods are not necessarily
directly comparable, their results can be expected to be close.

We use two different inner functions $g(x)=x$ and $g(x)=\sqrt[3]{x}$.
In the latter case, the nodes for integrating $f(x)$ are given
as $\lambda_i^3$. The first one gives an exact integral for function
$f(x)=x$ and the second one for $f(x)=\sqrt[3]{x}$.
The computations are performed using Matlab Symbolic Math
toolbox for computing $\boldsymbol{M}_n[g]$ as in \eqref{eq:non-orthonormal-matrix}
and the standard 64 bit IEEE 754
floating point numbers for the eigenvalue decomposition of matrix $\boldsymbol{M}_n[g]$.
The nodes and weights for the 5-point rules are shown in
Figure~\ref{fig:ex1-nodes-weights}.
The generalized quadrature points are taken from
\cite[Table 2]{Ma+Rokhlin+Wandzura:1996} where they
have been computed with the 128 bit Fortran (REAL*16) floating
point numbers and presented with 15 decimals.
\begin{figure}[htbp]
  \centering
  \includegraphics[width=\textwidth]{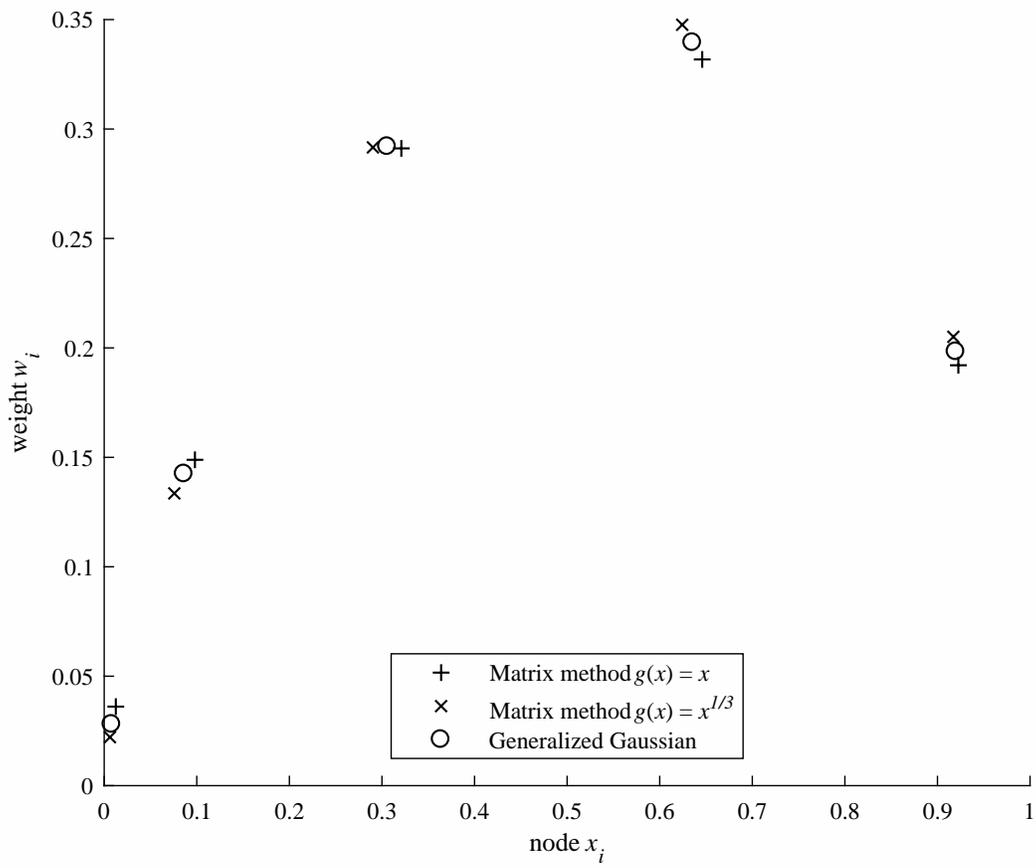}
  \caption{  \label{fig:ex1-nodes-weights}
    5-point quadrature rule nodes and weights.}
\end{figure}
We see from Figure~\ref{fig:ex1-nodes-weights}
that the three methods give nodes and weights that are
close to each other and that the generalized Gaussian quadrature nodes
and weights are located between the nodes and weights of the
proposed matrix methods. 

We compare the accuracy of the methods with a test function $f(x)=x^y$
where $y\in[0,6.5]$. Figure~\ref{fig:ex1-error5} shows the relative 
difference of
the exact solution to the quadrature
approximation with the three different 5-point rules.
\begin{figure}[htbp]
  \centering
  \includegraphics[scale=0.71]{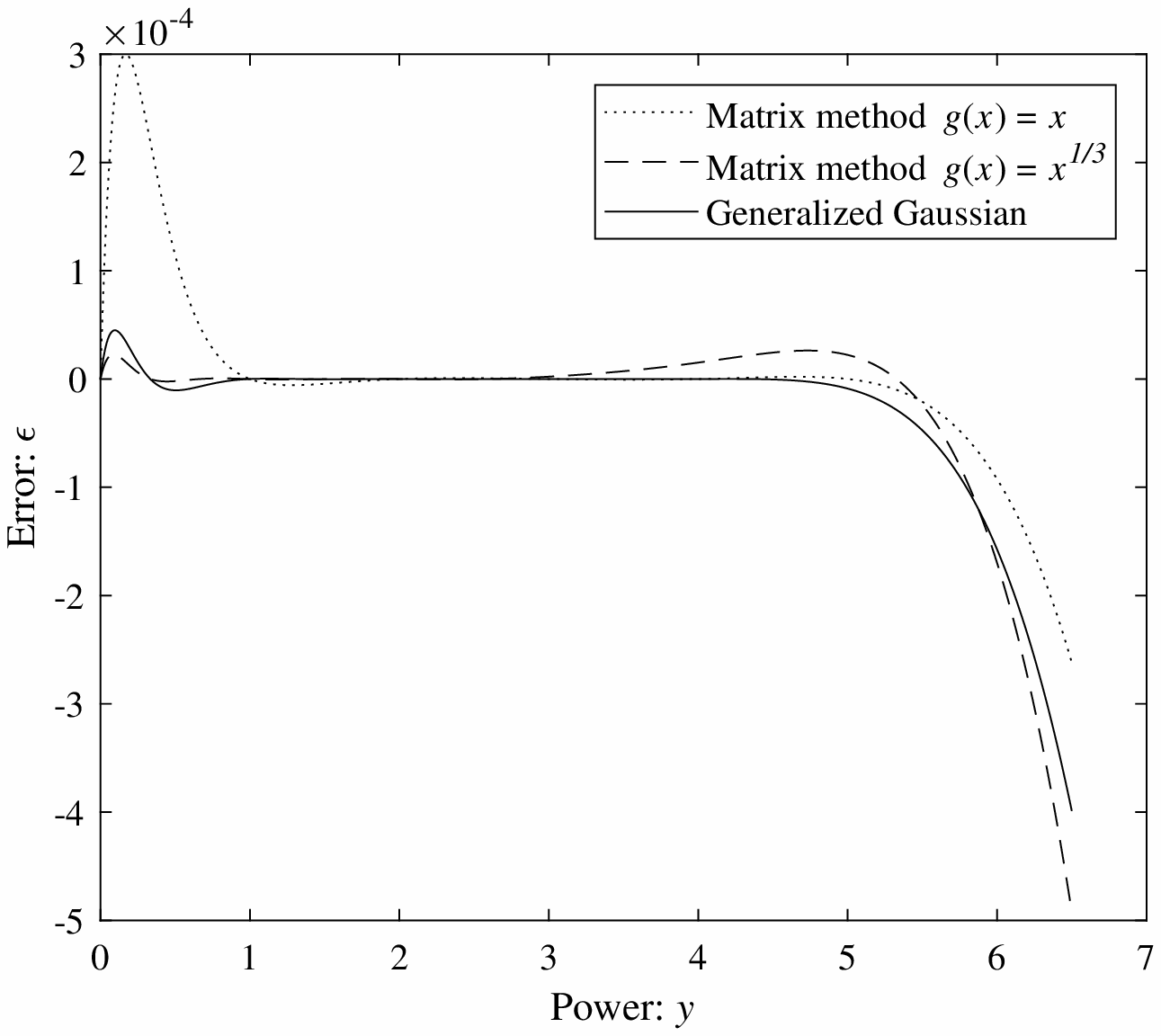}
  \caption{
    \label{fig:ex1-error5}
    Relative error
    $\epsilon=\frac{\sum\limits_{k=0}^4 w_k\,x_k^y}{\int_0^1 x^y\,dx}-1$ for the 5-point quadrature
 rules with $y\in [0,6.5]$.}
  \centering
  \includegraphics[scale=0.71]{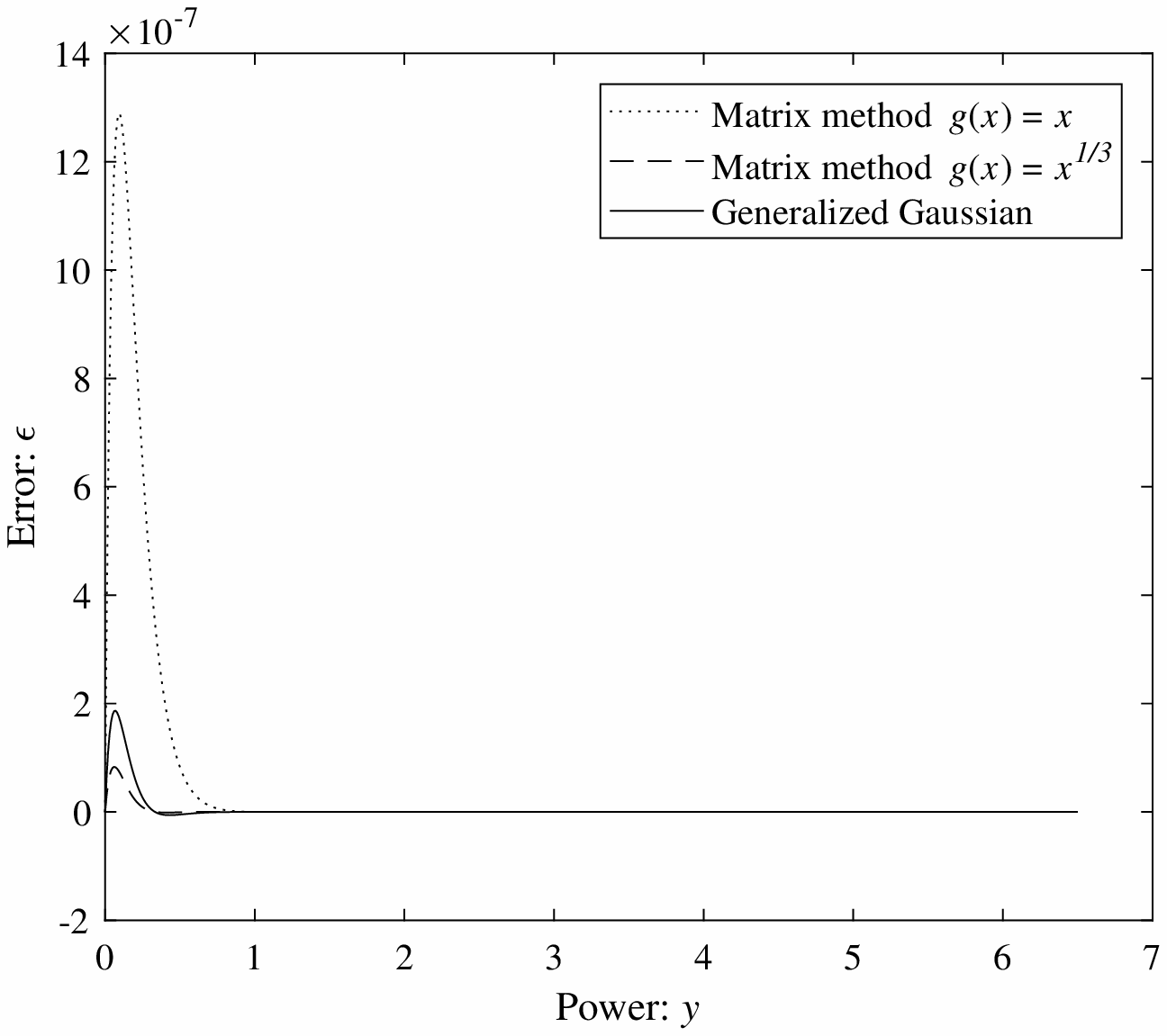}
  \caption{
    \label{fig:ex1-error20}
    Relative error
    $\epsilon=\frac{\sum\limits_{k=0}^{19} w_k\,x_k^y}{\int_0^1 x^y\,dx}-1$ for the
20-point quadrature rules with $y\in [0,6.5]$.}
\end{figure}
We see that the largest errors occur for the small powers and large powers
and, on those areas, the generalized Gaussian quadrature is between the matrix
method for $g(x)=x$ and $g(x)=\sqrt[3]{x}$. For small powers,
$g(x)=\sqrt[3]{x}$ gives the smallest error while for large powers it gives
the largest error.

In Figure~\ref{fig:ex1-error20}
we see the same comparison for the order 19 matrix method and the generalized
Gaussian quadrature
where the error is much smaller. Although not shown in
Figure~\ref{fig:ex1-error20},
the error of the three different methods behaves similarily for large
enough values of exponent $y$ as in Figure~\ref{fig:ex1-error5}.

The second integral example demonstrates the multiplicative nature
of the multiplication operator. The  integral is an integral of two variables
\begin{equation*}
  \int_0^1\int_0^1 e^{x\,y}\,\log(1+x+y)\,\,dx\,dy.
\end{equation*}
We select the linearly independent basis functions as
$\phi_0=1,\phi_1=x+y,\phi_2=x\,y,\phi_3=(x+y)^2,\phi_4=(x\,y)^2,\ldots,
\phi_{n-1}=(x+y)^k,\phi_{n}=(x\,y)^k$
and the inner functions for the multiplication operators
as $g_1=x\,y, g_2=x + y$. Then we can approximate the integral as
\begin{equation}
  \label{eq:int-ex-2}
  \int_0^1\int_0^1 e^{x\,y}\,\log(1+x+y)\,\,dx\,dy
  \approx
  \left[
    e^{\boldsymbol{M}_n[g_1]}\,\log(\boldsymbol{I}+\boldsymbol{M}_n[g_2])
    \right]_{0,0}.
\end{equation}
This approximation breaks the integral into two different numerical
approximations: one given by matrix
$\boldsymbol{M}_n[g_1]$ and the other by $\boldsymbol{M}_n[g_2]$.
These two matrix approximations are used for computing
approximations of multiplication operators for two
different functions: $e^{g_1}$
and $\log(1+g_2)$ respectively. The final integral approximation uses
also other matrix elements of the two multiplication operator approximations
besides the $(0,0)$ element as
\begin{equation*}
  \left[
    e^{\boldsymbol{M}_n[g_1]}\,\log(\boldsymbol{I}+\boldsymbol{M}_n[g_2])
    \right]_{0,0} =
  \sum_{k=0}^n
    \left[
    e^{\boldsymbol{M}_n[g_1]}
    \right]_{0,k} \,
      \left[
        \log(\boldsymbol{I}+\boldsymbol{M}_n[g_2])
    \right]_{k,0}.
\end{equation*}
Again, the matrix approximation of the multiplication operator
in terms of orthonormalized functions is computed symbolically.
Then we compute the approximation of the integral as a product
of the Matlab matrix functions in the 64 bit IEEE 754 floating point.
The results of the approximations are
shown in Table~\ref{tbl:int-ex-2} along with an error estimate. The error
estimate is based on a numerical approximation of the integral
with Matlab function quad2d. The error bound of the quad2d approximation
is $9.7553\cdot 10^{-12}$. 
\begin{table}[htbp]
  \caption{Numerical value of approximation \eqref{eq:int-ex-2}
    for $n+1$ linearly independent functions
    $\phi_0,\phi_1,\ldots,\phi_n$.}
  \label{tbl:int-ex-2}
  \begin{tabular}{|c|c|c|c|}
    \hline
    $n$  & $\phi_n$   & approximation & error \\
    \hline
    0 & $1$ & 8.900185973444169E-01 & 5.259050963568923E-02 \\
    1 & $x + y$ & 9.382241645325552E-01 & 4.384942447550944E-03 \\
    2 & $x \, y$ & 9.424586790473777E-01 & 1.504279327284586E-04 \\
    3 & $(x + y)^2$ & 9.424599771307293E-01 & 1.491298493768722E-04 \\
    4 & $x^2 \, y^2$ & 9.426178212955950E-01 & -8.714315488878022E-06 \\
    5 & $(x + y)^3$ & 9.426129095676246E-01 & -3.802587518419998E-06 \\
    6 & $x^3 \, y^3$ & 9.426094920018954E-01 & -3.850217892287233E-07 \\
    7 & $(x + y)^4$ & 9.426091679299925E-01 & -6.094988636018428E-08 \\
    8 & $x^4 \, y^4$ & 9.426091298353442E-01 & -2.285523803546852E-08 \\
    9 & $(x + y)^5$ & 9.426091128176409E-01 & -5.837534788888377E-09 \\
    10 & $x^5 \, y^5$ & 9.426091104398910E-01 & -3.459784903014906E-09 \\
    11 & $(x + y)^6$ & 9.426091075431513E-01 & -5.630451660465496E-10 \\
    12 & $x^6 \, y^6$ & 9.426091077121457E-01 & -7.320395400967072E-10 \\
    13 & $(x + y)^7$ & 9.426091069749081E-01 & 5.198064201294983E-12 \\
    14 & $x^7 \, y^7$ & 9.426091070047423E-01 & -2.463618198333961E-11 \\
    15 & $(x + y)^8$ & 9.426091069592208E-01 & 2.088529349464352E-11 \\
    16 & $x^8 \, y^8$ & 9.426091069628073E-01 & 1.729882903589441E-11 \\
    17 & $(x + y)^9$ & 9.426091069786899E-01 & 1.416200490211850E-12 \\
    18 & $x^9 \, y^9$ & 9.426091069789710E-01 & 1.135092020376760E-12 \\
    \hline
  \end{tabular}
\end{table}
We can see from the table that with about 15 or more basis functions,
the numerical integral converges to
value of about $0.942609107$ with an absolute error of less than $10^{-10}$.

\section{Conclusions}
\label{sec:conclusions}
We have introduced a method for numerically approximating integrals
as a matrix function of a matrix approximation of a multiplication
operator. We have also shown that the new method is a generalization
of Gaussian quadrature and that the new quadrature
method has similar properties as Gaussian quadrature. Additionally,
the convergence was proved for bounded functions.
The new method was numerically demonstrated in two examples.

\section*{Acknowledgements}
We thank the anonymous reviewers for valuable comments and
Toni Karvonen for assistance with generalized Gaussian quadrature.
The work was supported by Academy of Finland.
\bibliographystyle{elsarticle-num} 
\bibliography{references}

\end{document}